\documentclass[reqno,tbtags,intlimits,a4paper,oneside,12pt]{amsart}
\usepackage[cp1251]{inputenc}%
\usepackage[english]{babel}
\usepackage{amssymb,upref,calc,url}
\usepackage[mathscr]{eucal}
\usepackage{graphicx}
\usepackage{amsmath,amscd,amsthm,verbatim}
\usepackage[all]{xy}
\usepackage[in]{fullpage}
\newtheorem{thm}{\hskip\parindent Theorem}[section]
\newtheorem{lem}[thm]{\hskip\parindent Lemma}
\newtheorem{prb}[thm]{\hskip\parindent Problem}
\newtheorem{cor}[thm]{\hskip\parindent Corollary}

\theoremstyle{definition}
\newtheorem{dfn}{\hskip\parindent Definition}[section]

\newtheorem{rem}[dfn]{\hskip\parindent Remark}
\newcommand{\R}{\mathcal{R}}
\DeclareMathOperator{\wt}{wt}

\begin{document}

\title{Isomorphism of coefficient rings of~Buchstaber~and~Krichever~formal~groups
}
\author{E.\,Yu.~Bunkova}
\address{Steklov Mathematical Institute of Russian Academy of Sciences, Moscow, Russia}
\email{bunkova@mi-ras.ru}
\thanks{Supported by the Russian Foundation for Basic Research (Grant No. 20-01-00157 A)}

\begin{abstract}
In this work we give an explicit construction of the isomorphism of coefficient rings of Buchstaber and Krichever formal groups.
\end{abstract}

\maketitle

\begin{flushright}
{\it Dedicated to the Memory of I.M.~Krichever.} 
\end{flushright}

\section{Introduction}

The Buchstaber formal group was introduced in \cite{CC}. It's ring of coefficients was described in \cite{BU}. The Krichever formal group was introduced in \cite{BBKrich}. Both these formal groups have the same exponential, namely, the Krichever function that determines the Krichever genus \cite{Krichever-90}. There arises the question are this formal groups isomorphic, i.e. is there an isomorphism of their rings of coefficients  over $\mathbb{Z}$? The existence of such an isomorphism was proved in~\cite{Bak}. Here we construct this isomorphism explicitly and obtain the results on the rings of coefficients of~these formal groups that follow from this construction.

\section{Preliminaries and Definitions} \label{s11}

All rings considered in this paper will be supposed to be commutative rings with~unity.
 
\subsection{Formal groups}
\textit{A one-dimensional formal group law} (or, briefly, a \textit{formal group}) over a ring $R$ is a formal series
\begin{equation} \label{FGL}
F(u,v)=u+v+\sum\limits_{i,j\geqslant 1}a_{i,j}u^iv^j\in R[[u,v]]
\end{equation}
that satisfies the conditions of associativity
\begin{equation} \label{ass} 
F(F(u,v),w) = F(u,F(v,w))
\end{equation}
and commutativity \vspace{-5pt}
$$F(u,v)=F(v,u).$$
For the general theory of formal groups see \cite{Hazewinkel1978}.

The \emph{exponential of a formal group} $F(u, v)$ over $R$ is a formal series
$f(x) \in R \otimes \mathbb{Q}[[x]]$, determined by the initial conditions $f(0)=0, f'(0) = 1$ and the addition law
\begin{equation*}
f(x+y) = F(f(x),f(y)).
\end{equation*}
Thus, the substitution $u = f(x)$, $v = f(y)$ linearizes any formal group $F(u,v)$ over $R \otimes \mathbb{Q}$.
The series $g(u)$ functionally inverse to $f(x)$ is called \emph{logarithm of the formal group} $F(u,v)$.

In the ring $R$ one can select a subring $R_F$ of coefficients for the formal group $F(u,v)$.
By definition, the ring $R_F$ is multiplicatively generated by the coefficients $a_{i,j}$ in \eqref{FGL}.
In~the case $R = R_F$ the formal group $F(u,v)$ over $R$ is called a \emph{generating formal group}. 

A formal group $\mathcal{F}_U(u,v) = u + v + \sum \alpha_{i,j} u^i v^j$ over a ring $\R_U$ is called a
{\it universal formal group}
if for any formal group $F(u,v)$ over a ring $R$
there exists a unique ring homomorphism~$h_F\colon \R_U \to R$ such that $F(u,v) = u + v + \sum h_F(\alpha_{i,j}) u^i v^j$. 
The ring homomorphism~$h_F$ is called the \emph{classifying homomorphism}.
Further we suppose that the ring homomorphism $h_F$ is graded,
that is, the ring $R$ is graded and $\wt a_{i,j} = - 2 (i + j - 1)$.  
See Lemma 2.4 in \cite{BU} for the existence of the universal formal group.

When we consider formal groups of a given form, following \cite{BU}, we will speak of the {\it universal formal group of a given form}.
Namely, this is a formal group $\mathcal{F}_\#(u,v) = u + v + \sum \alpha_{i,j}^\# u^i v^j$ of this form over a ring $\R_\#$ such that for any formal group $F(u,v)$ of this form over a ring $R$
there exists a unique ring homomorphism~$h_\#\colon \R_\# \to R$ such that $F(u,v) = u + v + \sum h_F^\#(\alpha_{i,j}^\#) u^i v^j$.

\subsection{The Hirzebruch genus}
\text{}

Let $K$ be an associative commutative unital algebra over $\mathbb{Q}$.

The Hirzebruch genus is one of the most important classes of invariants of manifolds. A genus is a ring homomorphism
$L_f\colon \Omega_U \to K$ from the cobordism ring of stable complex manifolds $\Omega_U$ to $K$.
The Hirzebruch genus construction (see Sections 13 and 14 of \cite{CC}) matches such a genus to the series
\begin{equation} \label{fz}
f(z) = z + \sum_{k=1}^{\infty} f_k z^{k+1},  
\end{equation}
with coefficients~$f_k$ in~$K$.

There is an important property of $R$-integrability (see Section 14 of \cite{CC}),
that the image of $L_f$ belongs to the subring $R$ of the ring $K$.
There is a one-to-one correspondence between the set of formal groups over $R$ and the set of ring homomorphisms $\Omega_U \to R$.
Wherein the formal group $F$ corresponds to the ring homomorphism, determined by the Hirzebruch genus $L_f$, where $f$ is the exponential of this group with coefficients in $K = R \otimes \mathbb{Q}$.
Thus, to study the property of $R$-integrality of the Hirzebruch genus,
it is sufficient to 
construct the formal group corresponding to it and to study its ring of~coefficients.

\subsection{The Krichever function}
\label{s12} 
The classes of series for which the Hirzebruch genus possesses special properties are of~interest.
For example, the Krichever genus has the rigidity property~\cite{Krichever-90}
for manifolds with $S^1$-equivariant $SU$-structure. By \cite{B19} (see also Section E.5 of \cite{T} and \cite{Krichever-90}), one can define the the Krichever genus as a~Hirzebruch genus $L_f$ where the Krichever function $f(z)$ satisfies a differential equation of~the form
\begin{equation*}
f(z) f'''(z) - 3 f'(z) f''(z) = 6 q_1 f'(z)^2 + 12 q_2 f(z) f'(z) + 12 q_3 f(z)^2.
\end{equation*}
Classically (see \cite{Krichever-90}), the Krichever function  has the form 
\begin{equation} \label{fKr} 
{ \sigma(z) \sigma(\rho) \over \sigma(\rho - z)} \exp(\alpha z - \zeta(\rho) z),
\end{equation}
where $\sigma$ and $\zeta$ are Weierstrass functions (see \cite{Iso}, \cite{WW}).

\subsection{Buchstaber formal group} \label{s13}

The \emph{Buchstaber formal group} \cite{Buc90} over the ring $R$ is a formal group of the form 
\begin{equation} \label{FB}
F_B(u,v) =\frac{u^2 A(v) - v^2 A(u)}{u B(v) - v B(u)}, 
\end{equation}
where $A(u), B(u) \in R[[u]]$ and $A(0) = B(0) = 1$.
Let us introduce the notations
\begin{align*}
A(u) &= 1 + \sum_{k=1}^\infty A_k u^k, & B(u) &= 1 + \sum_{k=1}^\infty B_k u^k. 
\end{align*}
Note that the right hand side of \eqref{FB} does not depend on the coefficients $A_2$ and~$B_1$, therefore, further we set $A_2 = B_1 = 0$.
 
In~\cite{Buc90} (see also \cite[Theorem E~5.4]{T}) it was shown that the exponential of the universal Buchstaber formal group is the  function~\eqref{fKr}
(or it's degeneration).
The Buchstaber formal group is the universal formal group for the Krichever genus, that is, the exponential of the~universal Buchstaber formal group
is given by the Krichever function (see Section~\ref{s12}).

The ring of coefficients of the universal formal group of the form~\eqref{FB} is described in~\cite{BU}. Here we will give the results and notation from~\cite{BU} that we will need in the following sections.

\begin{lem}[Lemma 2.9 in \cite{BU}] \label{lr}
 Set $R_B = \mathbb{Z}[A_1, A_3, A_4, \ldots, B_2, B_3, B_4, \ldots]/I_{ass}^B$, where~$I_{ass}^B$ is the associativity ideal for the Buchstaber formal group \eqref{FB} generated by the relation \eqref{ass}. Then the formal group \eqref{FB} is~generating over $R_B$.
\end{lem}

By construction, the ring $R_B$ is the ring of coefficients of the universal Buchstaber formal group.

For natural numbers $m_1, \ldots, m_k$ we denote by $(m_1, \ldots, m_k)$ the greatest common divisor of these numbers. Using Euclidean algorithm, one can find integers $\lambda_1, \ldots, \lambda_k$ such that $\lambda_1 m_1 + \ldots \lambda_k m_k = (m_1, \ldots, m_k)$. For $n \geqslant 1$ set 
\[
 d(n+1) = \left(\begin{pmatrix}
                 n+1 \\ 1
                \end{pmatrix}, \begin{pmatrix}
                 n+1 \\ 2
                \end{pmatrix}, \ldots,
                \begin{pmatrix}
                 n+1 \\ n
                \end{pmatrix}
 \right).
\]
By Kummer's theorem (see Theorem 9.2 in \cite{BU}) we have
\[
 d(n+1) = \left\{ \begin{matrix}
                   p, & \text{ if } n+1 = p^k \text{ where } p \text{ is a prime},  \\
                   1 & \text{ if not.}
                  \end{matrix}
 \right.
\]
For each $n$ one can find a set $(\lambda_1, \ldots, \lambda_n)$ such that the equality holds
\[
 \lambda_1 \begin{pmatrix}
                 n+1 \\ 1
                \end{pmatrix} + \lambda_2 \begin{pmatrix}
                 n+1 \\ 2
                \end{pmatrix} + \ldots + \lambda_n \begin{pmatrix}
                 n+1 \\ n
                \end{pmatrix} = d(n+1).
\]

For the formal group \eqref{FGL} over a ring $R$ we introduce the element $e_n \in R$ using the~relation 
\begin{equation} \label{en}
 e_n = \lambda_1 a_{1,n} + \lambda_2 a_{2,n-1} + \ldots + \lambda_n a_{n,1}. 
\end{equation}
We introduce the ideal $I^2$ in the ring $R$ generated by the elements of the form $a_{i_1, j_1} a_{i_2, j_2}$, where $i_1, i_2, j_1, j_2 \geqslant 1$. In the case of a generating formal group $F$ over $R$, the ideal $I^2$ is~the ideal of decomposible elements of $R$. For $a \in R$ we denote by $\hat a$ the corresponding element $\hat a \in \hat R = R / I^2$.

\begin{dfn}[Definition 2.11 in \cite{BU}]
For a formal group $F$ over the ring $R$ the function $\rho: \mathbb{N} \to \{ \mathbb{N} \cup \infty \}$ is such that $\rho(n)$ is the order of the element $\hat e_n$ in the ring~$\hat R = R / I^2$.
\end{dfn}

\begin{rem}[Remark 4.4 in \cite{BU}]
For a generating formal group $F$ over $R$ the definition of the function $\rho$ does not depend on the coefficients in \eqref{en}. In this case the multiplicative generators of $R$ are the elements $e_n$ such that $\rho(n)>1$.
\end{rem}

\begin{cor}[Corollary 2.10 in \cite{BU}] \label{TBg}
The ring homomorphism $h_B: \mathcal{R}_U \to R_B$, classifying the Buchstaber formal group over $R_B$, is an epimorphism. One can select the set of multiplicative generators $\{e_1, \ldots, e_n, \ldots \}$ in $\mathcal{R}_U$ such that the set of non-zero elements $h_B(e_n)$ is a minimal set of generators of $R_B$.
\end{cor}

\begin{thm}[Theorem 6.1 in \cite{BU}] \label{tBU}
The ring $R_B$ is multiplicatively generated by the elements $h_B(e_n)$, where $n = 1, 2, 3, 4$, $n = p^r$ (where $r \geqslant 1$ and $p$ is prime), and $n = 2^k-2$ (where $k>2$). For the function $\rho(n) = \rho_B(n)$ we have
\[
\rho_B(n) = \left\{ \begin{matrix}
                     \infty & \text{ for } n = 1,2,3,4,\\ 
                     p & \text{ for } n = p^r, n \ne 2,3,4, \text{ and } r \geqslant 1 ,\\
                     2 & \text{ for } n = 2^k-2, k \geqslant 3,\\ 
                     1 & \text{ in all other cases.}\\ 
                    \end{matrix}
 \right.
\]
The ring $R_B$ has torsion only of order two.
\end{thm}

\subsection{Krichever formal group} \label{s6}

Let $\chi_1, \chi_2, \chi_3, \ldots \in R$.
The \emph{Krichever formal group} \cite{BBKrich} over $R$ is a formal group of~the~form 
\begin{equation} \label{FK}
F_K(u,v)=u b(v) + v b(u) - \chi_1 u v +\frac{b(u) \beta(u) - b(v) \beta(v)}{u b(v) - v b(u)}\,u^2 v^2,
\end{equation}
where $b(u), \beta(u) \in R[[u]]$ and 
\begin{align} \label{Kb}
b(u) &= 1 + \chi_1 u + \sum_{k=1}^\infty (\chi_{2 i} u^{2i} + 2 \chi_{2 i + 1} u^{2i+1}), \\ \label{Kbe} \beta(u) &= \sum_{k = 0}^\infty \left((k+1) \chi_{2k+2} u^{2k} + (2k+3) \chi_{2k+3} u^{2k+1}\right).
\end{align}

For the \emph{universal Krichever formal group} we consider the variables $\chi_1, \chi_2, \chi_3, \ldots$ as generators of the ring of coefficients.

\begin{lem}[Corollary 5.5 in \cite{BBKrich}] \label{lk}
 Set $R_K = \mathbb{Z}[\chi_1, \chi_2, \chi_3, \ldots]/I_{ass}^K$, where $I_{ass}^K$ is the associativity ideal for the Krichever formal group \eqref{FK} generated by the relation \eqref{ass}. Then the formal group \eqref{FK} is~generating over $R_K$.
\end{lem}

\begin{thm}[Theorem 5.6 in \cite{BBKrich}]
The exponential of a Krichever formal group is a Krichever function.
\end{thm}

It has been noted in \cite{Bak} that the rings of coefficients of universal Buchstaber formal group and universal Krichever formal group are isomorphic, that is $R_B = R_K$. In the next Section we give an independant proof of this result by constructing this isomorphism explicitly. It is
based on the explicit form~\eqref{FK} of Krichever formal group and Theorem~\ref{tBU}.

\section{Coefficient rings of universal Buchstaber~and~Krichever~formal~groups} \label{s2}

\begin{lem} \label{L1}
 $R_B \supset R_K$.
\end{lem}

\begin{proof}
By \eqref{FK} we have
\begin{equation} \label{FKB}
F_K(u,v)= \frac{u^2 \left( b(v)^2 - \chi_1 v b(v) - v^2 b(v) \beta(v)\right) - v^2 \left(b(u)^2 - \chi_1 u b(u) - u^2 b(u) \beta(u)\right)}{u b(v) - v b(u)}.
\end{equation}
Thus 
\[
F_K(u,v)= \frac{u^2 A_K(v) - v^2 A_K(u)}{u B_K(v) - v B_K(u)}
\]
is a Buchstaber formal group \eqref{FB}, where
\begin{align} \label{AB}
A_K(u) &= b(u)^2 - \chi_1 u b(u) - u^2 b(u) \beta(u) - \chi_2 u^2,\\
B_K(u) &= b(u) - \chi_1 u. \nonumber
\end{align}
Using the expressions  \eqref{Kb} and \eqref{Kbe}, we obtain from \eqref{AB} the relations for the coefficients of the series $A_K(u)$ and $B_K(u)$, where $n \geqslant 1$:
\begin{align}
 B_{2n} &= \chi_{2n}, \qquad
 B_{2n+1} = 2 \chi_{2n+1}, \qquad
 A_{1} = \chi_1, \label{cBK} \\
 A_{2n+1} &= (3-2n) \left(\chi_{2n+1} + \sum_{k=1}^{n} \chi_k \chi_{2n-k+1} \right)+ (n-2) \chi_1 \chi_{2n}, \nonumber\\
A_{4n} &= (1-n) \left( 2 \chi_{4n} + \sum_{k=1}^{2n-1} \chi_{2k} \chi_{4n-2k} + 4 \sum_{k=1}^{2n} \chi_{2k-1} \chi_{4n-2k+1} \right) + (4 n - 5) \chi_1 \chi_{4n-1}, \nonumber\\
A_{4n+2} &= (1-2n) \left( \chi_{4n+2}  + \sum_{k=1}^n \chi_{2k} \chi_{4n-2k+2} + 2 \sum_{k=1}^{2n+1} \chi_{2k-1} \chi_{4n-2k+3} \right) + (4n-3) \chi_1 \chi_{4n+1}, \nonumber
\end{align}
expressing all the coefficients $A_k$ and $B_k$ in $\chi_j$ for a Buchstaber formal group \eqref{FB} with~\eqref{AB}. From \eqref{FKB} we see that the associativity ideals $I_{ass}^B$ and $I_{ass}^K$ under the relations \eqref{cBK} coincide.
\end{proof}

\begin{thm}
$R_B = R_K$, i.e. the coefficient rings of the universal Buchstaber formal group and of the universal Krichever formal group are isomorphic over $\mathbb{Z}$.
\end{thm}

\begin{proof}
For the universal Buchstaber formal group \eqref{FB}, we denote by $\varepsilon_n$, $n = 1,2,\ldots$, the multiplicative generators $h_B(e_n)$ of the ring $R_B$, see Theorem \ref{TBg}. Denote by $\mathcal{P}$ the graded ring of polynomials in $\varepsilon_n$, $n = 1,2,\ldots$, with integer coefficients, and by $\mathcal{P}_n$ it's component of weight $n$. We have $A_n \in \mathcal{P}_n$, $B_n \in \mathcal{P}_n$.

For the universal Buchstaber formal group \eqref{FB}, the associativity ideal $I_{ass}^B$ is generated by the coefficients of the series (see \eqref{ass}):
\[
N_B(u,v,w) = F_B(F_B(u,v),w) - F_B(u,F_B(v,w)).
\]
The coefficient at $vw$ of the series $N_B(u,v,w)$ in $v, w$ as a function in $u$ is equal to
\begin{equation} \label{e12}
 B'(u) (B(u) + A_1 u) - 2  A_1 B(u) + 2 B_2 u + 2 {A(u)  - B(u)^2\over u}.
\end{equation}
Note that for $B(u) + A_1 u = 1 + A_1 u + \ldots$ the series $1/B(u)$ is a series in $u$ with coefficients in $\mathcal{P}$, and
\[
 {A(u)  - B(u)^2\over u} = A_1 - 2 B_2 u + (A_3 - 2 B_3) u^2 + (A_4 - 2 B_4 - B_2^2) u^3 + (A_5 - 2 B_5 - 2 B_2 B_3) u^4 + \ldots
\]
is a series in $u$ with coefficients in $\mathcal{P}$, therefore the expression
\[
 M(u) = {1 \over B(u) + A_1 u} \left(A_1 B(u) - B_2 u - {A(u)  - B(u)^2\over u}\right) = B_2 u + (2 B_3-A_3) u^2 + \ldots
\]
is a series in $u$ with coefficients in $\mathcal{P}$. Taking in account the grading, we denote $M(u) = \sum_{k \geqslant 1} M_{k+1} u^k$.
As the expression \eqref{e12} lies in the ideal $I_{ass}^B$, we~obtain the relation
\[
 B'(u) = 2 M(u).
\]
We have
\[
B'(u) = \sum_{k=1}^\infty k B_k u^{k-1} 
\]
Let us introduce the coefficients $C_n \in \mathcal{P}_n$ by the following relations for $n \geqslant 1$:
\begin{align}
 C_1 &= A_1, &
 C_{2n} &= B_{2n}, &
 C_{2n+1} &= M_{2n+1} - n B_{2n+1}. \label{cbm}
\end{align}
\end{proof}
Therefore for $n \geqslant 1$ we have
\begin{align*}
B_{2n} &= C_{2n}, & B_{2n+1} &= 2 C_{2n+1}, & A_1 &= C_1. 
\end{align*}
Comparing with \eqref{cBK}, we introduce the coefficients $S_k \in \mathcal{P}_k$ with $k \geqslant 3$ by the relations for $n \geqslant 1$:
\begin{align}
 S_{2n+1} &= A_{2n+1} - (3-2n) \left(C_{2n+1} + \sum_{k=1}^{n} C_k C_{2n-k+1} \right) - (n-2) C_1 C_{2n}, \nonumber\\
S_{4n} &= A_{4n} - (1-n) \left( 2 C_{4n} + \sum_{k=1}^{2n-1} C_{2k} C_{4n-2k} + 4 \sum_{k=1}^{2n} C_{2k-1} C_{4n-2k+1} \right) - (4 n - 5) C_1 C_{4n-1}, \nonumber\\
S_{4n+2} &= A_{4n+2} - (1-2n) \left( C_{4n+2}  + \sum_{k=1}^n C_{2k} C_{4n-2k+2} + 2 \sum_{k=1}^{2n+1} C_{2k-1} C_{4n-2k+3} \right) - (4n-3) C_1 C_{4n+1}. \nonumber
\end{align}
As the expression \eqref{e12} lies in the ideal $I_{ass}^B$, for all  $k \geqslant 3$ we obtain 
\[
 2 S_k = 0.
\]
Now let us prove that all $S_k = 0$ in $R_B$. Let $N$ be such that $S_N \ne 0$, $S_n = 0$ for $n<N$. By definition of $C_{2n+1}$ in \eqref{cbm}, we have $N \ne 2n+1$, thus $S_3 = 0$ and $A_3 =  C_{3}$.

The coefficient at $v^2 w$ of the series $N_B(u,v,w)$ in $v, w$ as a function in $u$ is equal to
\begin{multline} \label{e14}
{1 \over 2} B''(u) (A_1 u + B(u))^2
+ B'(u) \left(B(u) A_1 - u B_2\right) + B(u) \left(5 B_2 - A_1^2 \right) + \\+ 3 u \left(A_1 B_2 - A_3 + B_3\right)
- {(B'(u) u - 3 B(u)) (-B(u)^2 + A(u)) \over u^2} + A_1 {-4 B(u)^2 + A(u) \over u}.
\end{multline}
Here ${1 \over 2} B''(u) = \sum_{k=2}^{\infty} (k (k-1)/2) B_k u^{k-2}$ is a series in $u$ with coefficients in $\mathcal{P}$. The expression \eqref{e14} is a series in $u$ with coefficients in $\mathcal{P}$, and it's coefficient at $u^{N-2}$ gives an expression for $3 A_N$ in $C_1, C_2 \ldots, C_N$ for $N>3$. By the proof of Lemma \ref{L1} the corresponding relations hold for $3 S_N = 0$, and thus we obtain $S_N = 0$, which proves the~Theorem.

\begin{cor}
 For the universal Buchstaber formal group \eqref{FB} we have $B_{2n+1}\in 2 R_B$.
\end{cor}

\begin{cor}
 For the universal Buchstaber formal group \eqref{FB} we have $B'(u) \in 2 R_B[[u]]$.
\end{cor}

\begin{cor}
 For the universal Buchstaber formal group \eqref{FB} one can find a series ${B'(u) \over 2} \in R_B[[u]]$ such that $2 {B'(u) \over 2} = B'(u)$ and 
 the relation holds
\begin{align} \label{Brel}
A(u) &= (B(u) + A_1 u) \left(B(u) - u {B'(u) \over 2}\right) - B_2 u^2.
\end{align}
\end{cor}

\begin{cor}
For any Buchstaber formal group \eqref{FB} over a ring $R$ there exist $C_k \in R$ for $k = 1,2,3,\ldots$ such that the relations hold
\begin{align}
 B_{2n} &= C_{2n}, \qquad
 B_{2n+1} = 2 C_{2n+1}, \qquad
 A_{1} = C_1, \nonumber \\
 A_{2n+1} &= (3-2n) \left(C_{2n+1} + \sum_{k=1}^{n} C_k C_{2n-k+1} \right)+ (n-2) C_1 C_{2n}, \label{cB} \\
A_{4n} &= (1-n) \left( 2 C_{4n} + \sum_{k=1}^{2n-1} C_{2k} C_{4n-2k} + 4 \sum_{k=1}^{2n} C_{2k-1} C_{4n-2k+1} \right) + (4 n - 5) C_1 C_{4n-1}, \nonumber\\
A_{4n+2} &= (1-2n) \left( C_{4n+2}  + \sum_{k=1}^n C_{2k} C_{4n-2k+2} + 2 \sum_{k=1}^{2n+1} C_{2k-1} C_{4n-2k+3} \right) + (4n-3) C_1 C_{4n+1}. \nonumber
\end{align}
\end{cor}

\begin{cor} \label{cor32}
The generators $C_n$, $n \geqslant 1$ of the ring $R_B$ can be expressed from the relations \eqref{cB} by
\begin{align} \label{C1}
 C_{1} &= A_{1}, & C_{2n} = B_{2n}, 
\end{align}
and the recurrent relation
\begin{align} \label{C2}
 C_{2n+1} &= A_{2n+1} + (n-1) B_{2n+1} - (3-2n) \left(\sum_{k=1}^{n} C_k C_{2n-k+1} \right) - (n-2) C_1 C_{2n}. 
\end{align}
\end{cor}

These explicit formulas give unambigeous expressions for the multiplicative generators of the ring of coefficients of the Buchstaber formal group.

\begin{cor}[From Theorem \ref{tBU}]
The ring $R_B$ is multiplicatively generated by the elements $C_n$ determined in Corollary \ref{cor32}, where $n = 1, 2, 3, 4$, $n = p^r$ (where $r \geqslant 1$ and $p$ is prime), and $n = 2^k-2$ (where $k>2$). For~the~function $\rho(n) = \rho_B(n)$ we have
\[
\rho_B(n) = \left\{ \begin{matrix}
                     \infty & \text{ for } n = 1,2,3,4,\\ 
                     p & \text{ for } n = p^r, n \ne 2,3,4, \text{ and } r \geqslant 1 ,\\
                     2 & \text{ for } n = 2^k-2, k \geqslant 3,\\ 
                     1 & \text{ in all other cases.}\\ 
                    \end{matrix}
 \right.
\]
\end{cor}

\begin{cor}[From Theorem \ref{tBU}]
The ring $R_K$ is multiplicatively generated by the elements $\chi_n$, where $n = 1, 2, 3, 4$, $n = p^r$ (where $r \geqslant 1$ and $p$ is prime), and $n = 2^k-2$ (where $k>2$). For~the~function $\rho(n) = \rho_K(n)$ we have
\[
\rho_K(n) = \left\{ \begin{matrix}
                     \infty & \text{ for } n = 1,2,3,4,\\ 
                     p & \text{ for } n = p^r, n \ne 2,3,4, \text{ and } r \geqslant 1 ,\\
                     2 & \text{ for } n = 2^k-2, k \geqslant 3,\\ 
                     1 & \text{ in all other cases.}\\ 
                    \end{matrix}
 \right.
\]
\end{cor}

\begin{prb}
Express the torsion elements in $R_B$ in terms of $C_K$.
\end{prb}

\vfill
\eject

\end{document}